\documentclass[a4paper]{amsart}
\usepackage{amsmath}
\usepackage{amsthm}
\usepackage{amsfonts}
\usepackage{amscd}
\usepackage[online]{threeparttable}
\usepackage{hyperref}

\newcommand{\Z}{\mathbb{Z}}
\renewcommand{\H}{\mathrm{H}}

\newcommand{\Gm}{{\mathbb{G}_\mathrm{m}}}
\DeclareMathOperator{\Br}{Br}
\DeclareMathOperator{\Pic}{Pic}
\DeclareMathOperator{\Spec}{Spec}
\DeclareMathOperator{\coker}{coker}
\DeclareMathOperator{\Hom}{Hom}
\DeclareMathOperator{\Ext}{Ext}
\newcommand{\Picsh}{{\mathcal{P}\!\mathit{ic}}}
\newcommand{\YPsh}{\tilde{Y}_P}
\newcommand{\XPsh}{\tilde{X}_P}
\newcommand{\tors}{{\mathrm{tors}}}
\newcommand{\et}{{\textrm{\'et}}}
\newcommand{\E}{\mathbf{E}}

\newcommand{\ns}{\mathrm{ns}}

\newtheorem{theorem}{Theorem}
\newtheorem{proposition}[theorem]{Proposition}
\newtheorem{corollary}[theorem]{Corollary}

\title{Brauer groups of singular del Pezzo surfaces}
\author{Martin Bright}
\address{Department of Mathematics \\
American University of Beirut \\
Bliss Street \\ Hamra \\ Beirut \\ Lebanon}
\email{martin.bright@aub.edu.lb}

\begin{document}

\begin{abstract}
We describe the effect of rational singularities on the Brauer
group of a surface, and compute the Brauer groups of all singular del
Pezzo surfaces over an algebraically closed field.
\end{abstract}

\maketitle

\section{Introduction}

The Brauer group of a variety $X$, which in this paper we take to mean
the cohomology group $\Br X = \H^2(X,\Gm)$, was extensively studied by
Grothendieck~\cite{Grothendieck:GB}.  Brauer groups of singular
varieties are not particularly well behaved: in particular, the Brauer
group of a singular variety need not inject into the Brauer group of
its function field.  The purely local question, of understanding the
Brauer group of the local ring of a singularity, has been well
studied: see, for example,~\cite{DFM:JA-1993}.  An interesting feature
of the results discussed in this article is that the calculation is a
global one, and often leads to elements of the Brauer group which are
locally trivial in the Zariski topology.  One individual example of
such an element was given by Ojanguren~\cite{Ojanguren:JA-1974}, whose
algebra is of order 3 and defined on a singular cubic surface with
three $A_2$ singularities; it will be shown below that this is the
only type of singular cubic surface admitting a 3-torsion Brauer
element.  A more general framework for studying such examples,
described by Grothendieck in~\cite{Grothendieck:GB}, was developed by
De~Meyer and Ford~\cite{DF:AAAM-1990} to give examples of toric
surfaces admitting non-trivial, locally trivial Azumaya algebras.

In this article we take a slightly different approach which, for
varieties with rational singularities, shows how the calculation of
the Brauer group can be made very explicit by using the intersection
pairing.  We then apply this to arguably the simplest interesting
class of singular projective surfaces, namely the singular del Pezzo
surfaces.  These are easy to approach for two reasons: they have
rational singularities; and they come with a natural desingularisation
which is a rational surface.  In Proposition~\ref{prop:br} we show how
to combine the Leray spectral sequence for the desingularisation with
Lipman's detailed description of the local Picard groups above the
singular points \cite{Lipman:IHES-1969}.  In particular, it follows
that the Brauer group may be easily computed using the intersection
form on the desingularisation.  For singular del Pezzo surfaces, this
is well understood, and in section~\ref{sec:dp} we apply
Proposition~\ref{prop:br} to compute the Brauer groups of all singular
del Pezzo surfaces over an algebraically closed field; the Brauer
group depends only on the singularity type of the surface.  The
arguments, and hence the results, are valid in arbitrary
characteristic.

The principal motivation for this article is in applying the Brauer
group to study rational points of del Pezzo surfaces, as first
suggested by Manin~\cite{Manin:GBG}.  For arithmetic questions, it is
often more useful to work with a desingularisation of the original
variety, and so the Brauer group of the singular variety is not of
obvious interest.  However, there are some situations where one cannot
avoid looking at the Brauer group of a singular variety; the situation
we have in mind is that of a model of a del Pezzo surface over a local
ring, where the Brauer group of the (possibly singular) special fibre
must be taken into account.

\section{The Brauer group of a surface with rational singularities}

In this section we study the Brauer group of a surface $Y$ having only
isolated rational singularities, over an algebraically closed base
field.
Following Lipman, by a \emph{desingularisation} $X \to Y$ we
mean a proper birational morphism from a regular scheme $X$.  If $Y$
is a normal surface with finitely many rational singularities, then
there is a unique minimal desingularisation $X \to Y$ which may be
constructed as a sequence of blow-ups at singular points.

We write $\Br(X/Y)$ to mean $\ker(\Br Y \to \Br X)$.  If $Y$ is
integral with function field $K$, then there is a sequence of maps $\Br Y
\to \Br X \to \Br K$.  Since $X$ is regular, $\Br X$ injects into $\Br
K$; it follows that $\Br(K/Y) \cong \Br(X/Y)$.

Whenever $A$ is an Abelian group, $A^*$ denotes the group $\Hom(A,\Z)$.

\begin{proposition}\label{prop:br}
Let $Y$ be a normal surface over an algebraically closed field $k$; suppose that $Y$ has finitely many rational singularities, and let $f\colon X \to Y$ be the minimal desingularisation.  Let $\E$ denote the subgroup of $\Pic X$ generated by the classes of the exceptional curves of the resolution, and let $\theta \colon \Pic X \to \E^*$ be the homomorphism induced by the intersection pairing on $\Pic X$.  Then there is an exact sequence
\[
0 \to \Pic Y \xrightarrow{f^*} \Pic X \xrightarrow{\theta} \E^* \to \Br Y \xrightarrow{f^*} \Br X.
\]
\end{proposition}

\begin{proof}
Since $f$ is proper and birational, we have $f_* \Gm = \Gm$.  It follows that,
for any flat morphism of schemes $Y' \to Y$, if $f_{Y'} \colon X \times_Y
Y' \to Y'$ denotes the base change of $f$, the following sequence is
exact (see~\cite[Section~8.1, Proposition~4]{BLR:NM}):
\begin{equation}\label{eq:leray}
0 \to \Pic Y' \xrightarrow{f^*_{Y'}} \Pic (X \times_Y Y') \to \Picsh_{X/Y}(Y') \to \Br Y' \xrightarrow{f^*_{Y'}} \Br (X \times_Y Y').
\end{equation}

Taking $Y'=Y$ in~\eqref{eq:leray} gives the exact sequence
\begin{equation}\label{eq:step0}
0 \to \Pic Y \xrightarrow{f^*} \Pic X \to \Picsh_{X/Y}(Y) \to \Br Y \xrightarrow{f^*} \Br X.
\end{equation}
So it will be enough to exhibit an isomorphism $\alpha \colon \Picsh_{X/Y}(Y) \to \E^*$ such that composing $\alpha$ with the
natural homomorphism $\Pic X \to \Picsh_{X/Y}(Y)$ gives the homomorphism
$\theta$ described in the statement of the theorem.
From now on, we work with
$\Picsh_{X/Y}$ as a sheaf only on the small \'etale site of $Y$, in order
to be able to talk about its stalks.

\paragraph{Step 1:  Localisation} 
As the sheaf $\Picsh_{X/Y}$ on
$Y_\et$ is supported on the singular points, the natural map
\begin{equation}\label{eq:stalks}
\Picsh_{X/Y}(Y) \to \prod_{P \text{ singular}}(\Picsh_{X/Y})_P,
\end{equation}
from the global sections of $\Picsh_{X/Y}$ to the direct product of its stalks at
the singular points, is an isomorphism.


At each singular point $P$, let $\YPsh$ denote $\Spec
\mathcal{O}_{Y,P}^{\textrm{sh}}$, the spectrum of the Henselisation of the
local ring at $P$, and set $\XPsh = X \times_Y \YPsh$.  The stalk
of $\Picsh_{X/Y}$ at the geometric point $P$ is naturally isomorphic
to $\Picsh_{X/Y}(\YPsh)$ which is simply $\Pic \XPsh$, as is seen by
taking $Y' = \YPsh$ in~\eqref{eq:leray} and using the facts that $\Pic
\YPsh$ and $\Br \YPsh$ are trivial (for the latter, see~\cite[IV,
  Corollary~1.7]{Milne:EC}).  Combining this with the
isomorphism~\eqref{eq:stalks}, we see that the natural map
\begin{equation*}\label{eq:step1}
\Picsh_{X/Y}(Y) \to \prod_P \Pic \XPsh
\end{equation*}
is an isomorphism.



\paragraph{Step 2: Lipman's description of $\Pic\XPsh$} For each singular point $P$ of $Y$, denote by $\E_P$ the subgroup of $\Pic X$ generated by the exceptional curves lying over $P$.  Let $Y_P$ denote the spectrum of the Zariski local ring of $Y$ at $P$, and $X_P = X \times_Y Y_P$.  We will use $\theta_P$ to denote the homomorphism $\Pic X_P \to \E_P^*$ induced by the intersection pairing on $X$.\footnote{Lipman's definition of the map $\theta_P$ is slightly more general, involving dividing by the least degree of an invertible sheaf on each exceptional curve.  Since we are working over an algebraically closed field, all of our exceptional curves have a $k$-point, hence an invertible sheaf of degree $1$.}
Lipman~\cite[Part~IV]{Lipman:IHES-1969} studied the kernel and cokernel of $\theta_P$ in detail, defining an exact sequence
\begin{equation*}\label{eq:lipman}
0 \to \Pic^0 X_P \to \Pic X_P \xrightarrow{\theta_P} \E_P^* \to G(Y_P) \to 0
\end{equation*}
attached to the resolution $X_P \to Y_P$, and showed that $\Pic^0
X_P=0$ when $Y_P$ has a rational singularity, and that $G(Y_P)=0$ when
$Y_P$ is Henselian.  We thus obtain isomorphisms $\Pic\XPsh \cong
\E_P^*$, such that the composite homomorphism $\Pic X \to \prod_P \Pic\XPsh
\to \prod_P \E_P^*$ is $\theta_P$.

\paragraph{Step 3: Globalisation} Finally, note that two exceptional
curves lying above distinct singularities of $Y$ are disjoint, so in
particular have intersection number zero.  Therefore the subgroups
$\E_P \subseteq \Pic X$ are mutually orthogonal, and so $\E \cong \bigoplus_P \E_P$ and $\E^* \cong \prod_P \E_P^*$.

It is now easily verified that replacing $\Picsh_{X/Y}(Y)$
in~\eqref{eq:step0} with $\E^*$ does indeed lead to the desired exact
sequence.
\end{proof}

\begin{corollary}\label{cor:local}
If $P$ is a singular point of $Y$, then $\Br(X_P/Y_P)$ is isomorphic to the
cokernel of $\theta_P \colon \Pic X \to \E_P^*$, which is equal to
Lipman's group $G(Y_P)$.
\end{corollary}
\begin{proof}
Applying the proposition to $Y_P$ shows that $\Br(X_P/Y_P)$ is
isomorphic to $\coker(\Pic X_P \to \E_P^*)$, which by definition is
equal to $G(Y_P)$.  Since $X$ is smooth, the restriction map $\Pic X
\to \Pic X_P$ is surjective, and the statement follows.
\end{proof}

\begin{corollary}\label{cor:one}
If $Y$ has only one singularity $P$, then $\Br(X/Y) \cong
\Br(X_P/Y_P)$.
\end{corollary}
\begin{proof}
In this case $\E = \E_P$, so the statement follows immediately from
Corollary~\ref{cor:local}.
\end{proof}

\section{Singular del Pezzo surfaces}
\label{sec:dp}

In this section we apply Proposition~\ref{prop:br} to compute the Brauer
groups of singular del Pezzo surfaces.  We refer
to~\cite{CT:PLMS-1988} and~\cite{Demazure:SDP} for background details on
singular del Pezzo surfaces.  

Let $X$ be a generalised del Pezzo surface over an algebraically
closed field $k$, and $f\colon X \to Y$ the morphism contracting the $(-2)$-curves (and nothing else), so that $Y$ is the corresponding singular del Pezzo surface.  The Picard group
of $X$ fits into a short exact sequence
\[
0 \to Q \to \Pic X \xrightarrow{(\cdot,K_X)} \Z \to 0
\]
where $Q$ is the subgroup orthogonal to the canonical class $K_X$
under the intersection pairing.  The exceptional curves of $f$ are all
contained in $Q$.  Let $\E$
denote the subgroup of $Q$ generated by all the exceptional curves of
$X \to Y$ (equivalently, all the $(-2)$-curves on $X$).

\begin{proposition}\label{prop:brdp}
$\Br Y$ is isomorphic to $(Q/\E)_\tors$.
\end{proposition}
\begin{proof}
Firstly, $\Br X$ is trivial, for $X$ is a rational surface.  By Proposition~\ref{prop:br}, $\Br Y$ is therefore isomorphic to the cokernel of the map $\theta \colon \Pic X \to \E^*$.  Now $\theta$ factors as $\Pic X \to Q^* \to \E^*$,
giving an exact sequence
\[
\coker(\Pic X \to Q^*) \to \Br Y \to \coker(Q^* \to \E^*) \to 0.
\]
It follows from the description of $Q$ in~\cite[II.4]{Demazure:SDP} that $\Pic X
\xrightarrow{\theta} Q^*$ is surjective.  Indeed, one easily checks
that the basis of $Q$ given by the simple roots $\alpha_i$ described there
can be extended (for example, by adjoining one exceptional class
$E_1$) to a basis of $\Pic X$.  So we are left with an isomorphism
between $\Br Y$ and $\coker(Q^* \to \E^*)$.  To compute the latter
group, we take the short exact sequence
\[
0 \to \E \to Q \to (Q/\E) \to 0
\]
and apply $\Hom(\cdot, \Z)$ to obtain the longer exact sequence
\[
0 \to (Q/\E)^* \to Q^* \to \E^* \to \Ext^1(Q/\E,\Z) \to
\Ext^1(Q,\Z).
\]
As $Q$ is a free Abelian group, we have $\Ext^1(Q,\Z)=0$, and
therefore $\Br Y$ is isomorphic to $\Ext^1(Q/\E,\Z)$, which by
a standard calculation is isomorphic to $(Q/\E)_\tors$.
\end{proof}

We note the following interesting corollary.
\begin{corollary}
Let $Y$ be a singular del Pezzo surface over an algebraically closed
field, and denote by $Y^\ns$ the non-singular locus of $Y$.  Then
there is an isomorphism of abstract groups $\Br Y \cong
\Pic(Y^\ns)_\tors$.
\end{corollary}
\begin{proof}
Since $Y^\ns$ is isomorphic to the complement of the exceptional
curves in $X$, we have $\Pic Y^\ns \cong (\Pic X)/\E$ and so
$\Pic(Y^\ns)_\tors \cong (Q/\E)_\tors$.
\end{proof}

It remains to enumerate the possible singularity types of del Pezzo
surfaces and to compute $Q/\E$ in each case.  The algorithm for
listing the possible configurations of $(-2)$-curves is well known, as
is the list of possible configurations, so we only summarise the
algorithm very briefly.  The free Abelian group $Q$, together with the
negative definite intersection pairing, is isomorphic to the root
lattice of a particular root system depending only on the degree of
the surface.  Within this root lattice, the exceptional divisors of
the desingularisation $X \to Y$ form a set of simple roots in some
sub-root system, and indeed form a $\Pi$-system in the sense of
Dynkin~\cite[\S 5]{Dynkin:AMST-1957}.  To list the $\Pi$-systems
contained in $Q$, we use the following two theorems
from~\cite{Dynkin:AMST-1957}:
\begin{itemize}
\item Theorem~5.2: every $\Pi$-system is contained in a $\Pi$-system which is of
  maximal rank, that is, which spans $Q$ as a vector space;
\item Theorem~5.3: the $\Pi$-systems of maximal rank may be all be obtained from
  some set of simple roots in $Q$ by iterating the following
  procedure, called an \emph{elementary transformation}: starting with a set of simple roots, choose one connected
  component of the associated Dynkin diagram; adjoin the most
  negative root of that component and discard one of the original
  simple roots of that component.
\end{itemize}
So, starting from any choice of simple roots in $Q$, we can obtain all
$\Pi$-systems up to the action of the Weyl group.  Not quite all of
these can actually be achieved as configurations of $(-2)$-curves: see~\cite{Urabe:S1981}, though it is not immediately clear that the methods there also apply in positive characteristic.  

Let us remark that, given a root system $R$, the primes dividing
$\#(\Z R/\Z R')_\tors$ for $R'$ a closed subsystem of $R$ are called
\emph{bad primes}: see, for example, \cite[Appendix~B]{MT:LAG}.  A
corollary of Proposition~\ref{prop:brdp} is that the primes which can
divide the order of the Brauer group of a singular del Pezzo surface
of degree $d$ are the bad primes of the associated root system.  It
turns out that the bad primes are simply those occurring as
coefficients when a maximal root is expressed in terms of simple
roots, and so they are easily listed.  There are no bad primes for
$A_n$; $2$ is the only bad prime for $D_n$ ($n \ge 4$); $2$ and $3$
are the bad primes for $E_6$ and $E_7$; and $E_8$ has bad primes $2$,
$3$ and $5$.

\begin{theorem}
Let $Y$ be a singular del Pezzo surface of degree $d$ over an
algebraically closed field.  If $d \ge 5$, then $\Br Y = 0$.  If $1
\le d \le 4$, then the Brauer group of $Y$ is determined by its
singularity type; the singularity types giving rise to non-trivial
Brauer groups are listed in Tables~\ref{table:deg4}--\ref{table:deg1}.
Each class in $\Br Y$ is represented by an Azumaya algebra.  Except
for the singularity types $A_7$ in degree 2, and $A_7$, $A_8$ and
$D_8$ in degree 1, the corresponding Azumaya algebras are locally
trivial in the Zariski topology.
\end{theorem}
\begin{proof}
For $d \ge 5$, the relevant root system is of type $A_n$, so there are
no bad primes and the Brauer group is trivial.  For $1 \le d \le 4$,
the results of applying the algorithm described above are listed in
the tables.  Since $\Br Y$ is torsion, it follows from a result proved
by Gabber and, independently, by de~Jong~\cite{DeJong:Gabber} that
every class is represented by an Azumaya algebra.  It remains to prove
the statement about Zariski-local triviality.  If $P$ is a singular
point of a singular del Pezzo surface $Y$ then
Corollary~\ref{cor:local} shows that, in the notation used there,
$\Br(X_P/Y_P) \cong \coker(\Pic X \to \E_P^*)$; it is enough to
show that $\Br(X_P/Y_P) = 0$.  Replacing $Y$ by a del Pezzo surface of
the same degree, but with only one singularity of the same type as
$P$, changes neither $\Pic X$, $\E_P^*$ nor the map between them, so
we may assume that $P$ is the only singularity of $Y$.  Then
$\Br(X_P/Y_P) = \Br(X/Y) = \Br Y$ by Corollary~\ref{cor:one}.  But the
tables show that $\Br Y=0$, except in the cases listed above.
\end{proof}

\begin{table}[p]
\caption{Brauer groups of singular del Pezzo surfaces of degree $4$}
\begin{tabular}{ll|ll}
Singularity type & Brauer group &
Singularity type & Brauer group \\
\hline
$2A_1 + A_3$ & $\Z/2\Z$ &
$4A_1$ & $\Z/2\Z$ 
\end{tabular}
\label{table:deg4}
\end{table}

\begin{table}[p]
\caption{Brauer groups of singular del Pezzo surfaces of degree $3$}
\begin{tabular}{ll|ll}
Singularity type & Brauer group &
Singularity type & Brauer group \\
\hline
$A_1 + A_5$ & $\Z/2\Z$ &
$2A_1 + A_3$ & $\Z/2\Z$ \\
$4A_1$ & $\Z/2\Z$ & $3A_2$ & $\Z/3\Z$
\end{tabular}
\end{table}

\begin{table}[p]
\caption{Brauer groups of singular del Pezzo surfaces of degree $2$}
\begin{threeparttable}
\begin{tabular}{ll|ll}
Singularity type & Brauer group &
Singularity type & Brauer group \\
\hline
$A_1 + 2A_3$ & $\Z/4\Z$ &
$5A_1$ & $\Z/2\Z$ \\
$A_1 + A_5$ & $\Z/2\Z$ &
$6A_1$ & $(\Z/2\Z)^2$ \\
$A_1 + D_6$ & $\Z/2\Z$ &
$7A_1$ $\dagger$ & $(\Z/2\Z)^3$ \\
$2A_1 + A_3$ & $\Z/2\Z$ &
$A_2 + A_5$ & $\Z/3\Z$ \\
$2A_1 + D_4$ & $\Z/2\Z$ &
$3A_2$ & $\Z/3\Z$ \\
$3A_1 + A_3$ & $\Z/2\Z$ &
$2A_3$ & $\Z/2\Z$ \\
$3A_1 + D_4$ & $(\Z/2\Z)^2$ &
$A_7$ & $\Z/2\Z$ \\
$4A_1$ * & $\Z/2\Z$
\end{tabular}
\begin{tablenotes}
\item [*] There are (up to the action of the Weyl group) two different
  ways of embedding $4A_1$ into $E_7$, and so two different
  singularity types of degree $2$ del Pezzo surface with root system
  $4A_1$.  One of these has Brauer group $\Z/2\Z$; the other has
  trivial Brauer group.
\item [$\dagger$] This sub-root system does not arise from a del Pezzo surface~\cite{Urabe:S1981}.
\end{tablenotes}
\end{threeparttable}
\end{table}

\begin{table}[p]
\caption{Brauer groups of singular del Pezzo surfaces of degree $1$}
\begin{threeparttable}
\begin{tabular}{ll|ll}
Singularity type & Brauer group &
Singularity type & Brauer group \\
\hline
$A_1+A_2+A_5$ & $\Z/6\Z$ &
$4A_1+A_3$ & $(\Z/2\Z)^2$ \\
$A_1+3A_2$ & $\Z/3\Z$ &
$4A_1+D_4$ $\dagger$ & $(\Z/2\Z)^3$ \\
$A_1+2A_3$ & $\Z/4\Z$ &
$5A_1$ & $\Z/2\Z$ \\
$A_1+A_5$ * & $\Z/2\Z$ &
$6A_1$ & $(\Z/2\Z)^2$ \\
$A_1+A_7$ & $\Z/4\Z$ &
$7A_1$ $\dagger$ & $(\Z/2\Z)^3$ \\
$A_1+D_6$ & $\Z/2\Z$ &
$8A_1$ $\dagger$ & $(\Z/2\Z)^4$ \\
$A_1+E_7$ & $\Z/2\Z$ &
$A_2+A_5$ & $\Z/3\Z$ \\
$2A_1+A_2+A_3$ & $\Z/2\Z$ &
$A_2+E_6$ & $\Z/3\Z$ \\
$2A_1+A_3$ * & $\Z/2\Z$ &
$3A_2$ & $\Z/3\Z$ \\
$2A_1+2A_3$ & $\Z/2\Z \times \Z/4\Z$ &
$4A_2$ & $(\Z/3\Z)^2$ \\
$2A_1+A_5$ & $\Z/2\Z$ &
$A_3+D_4$ & $\Z/2\Z$ \\
$2A_1+D_4$ & $\Z/2\Z$ &
$A_3+D_5$ & $\Z/4\Z$ \\
$2A_1+D_5$ & $\Z/2\Z$ &
$2A_3$ * & $\Z/2\Z$ \\
$2A_1+D_6$ & $(\Z/2\Z)^2$ &
$2A_4$ & $\Z/5\Z$ \\
$3A_1+A_3$ & $\Z/2\Z$ &
$A_7$ * & $\Z/2\Z$ \\
$3A_1+D_4$ & $(\Z/2\Z)^2$ &
$A_8$ & $\Z/3\Z$ \\
$4A_1$ * & $\Z/2\Z$ &
$2D_4$ & $(\Z/2\Z)^2$ \\
$4A_1+A_2$ & $\Z/2\Z$ &
$D_8$ & $\Z/2\Z$
\end{tabular}
\label{table:deg1}
\begin{tablenotes}
\item [*] Each of these five root systems may be embedded into $E_8$ in two
  distinct ways.  In all five cases, one way results in trivial Brauer
  group; the other results in the Brauer group shown in the table.
\item [$\dagger$] These sub-root systems do not arise from del Pezzo surfaces~\cite{Urabe:S1981}.
\end{tablenotes}
\end{threeparttable}
\end{table}

\bibliographystyle{abbrv}
\bibliography{brsing}

\end{document}